\documentclass[12pt]{article}
\usepackage{amsmath}
\usepackage{amssymb}
\usepackage{hyperref}
\usepackage{color}
\usepackage{graphicx}
\usepackage[left=3cm,top=3cm,right=3cm,bottom=3cm]{geometry}
\newtheorem{theorem}{Theorem}
\newtheorem{lemma}[theorem]{Lemma}

\newtheorem{corollary}[theorem]{Corollary}

\newcommand{\by}{\hspace{0.02in}{\scriptstyle{\square}}\hspace{0.02in}}
\newcommand{\noin}{\noindent}
\usepackage{setspace}
\usepackage{epsfig}
\newcommand{\qed}{\ \hfill \rule{1ex}{1ex}\vspace{0.1in}} 
\newenvironment{proof}{{\noin \bf Proof}: }{\qed}

\usepackage{comment}

\newcommand{\zombie}{{\scalebox{0.60}{$\mathcal{Z}$}}}

\begin{document}

\title{The Game of Zombies and Survivors on Cartesian Products of Graphs}

\author{S.L. Fitzpatrick\footnote{School of  Mathematical and Computational Sciences, University of Prince Edward Island}}

\maketitle

\begin{abstract}
We consider the game of Zombies and Survivors as introduced by Fitzpatrick, Howell, Messinger and Pike (2016) This is a variation of the game Cops and Robber where the zombies (in the cops' role) are of limited intelligence and will always choose to move closer to a survivor (who takes on the robber's role). The zombie number of a graph is defined to be the minimum number of zombies required to guarantee the capture of a survivor on the graph.  In this paper, we show that the zombie number of the Cartesian product of $n$ non-trivial trees is exactly $\lceil 2n/3 \rceil$.  This settles a conjecture by Fitzpatrick et. al. (2016) that this is the zombie number for the $n$-dimensional hypercube.  In proving this result, we also discuss other variations of Cops and Robber involving active and flexible players.  \end{abstract}

\section{Introduction}

\subsection{Active Cops, Flexible Robbers and Zombies}

 The game of {\it Cops and Robber} is played on a graph with  two players: a set of {\it cops} and a single {\it robber}. Initially, each cop chooses a vertex to occupy, and then the robber chooses a vertex to occupy.   The cops and robber then alternate moves.  For the cops' move, each of the cops either moves to an adjacent vertex or stays at her current location (referred to as a {\it pass}).  The robber's move is defined similarly.  The pair of moves by the cops and the robber is referred to as a {\it round}, with the initial choice of positions considered to be round zero. The cops win if, after some finite number of rounds, a cop occupies the same vertex as the robber.  Otherwise, the robber wins. 
 For a graph $G$, the {\it cop number}, $c(G)$, is the minimum number of cops required to win in  $G$.  For a review of the game of Cops and Robber see \cite{AMS, book}.

In the game of Cops and Robber, both the cops and the robber have the option to pass on any move.  A variation of the game have been consider where both the cops and robber are {\it active}.    In the active game, as introduced by Aigner and Fromme \cite{AF},  both the robber and a non-empty subset of the cops move to adjacent vertices on their respective turns. This variation was  investigated further by To\u{s}i\'{c} \cite{Tosic} and Neufeld \cite{N}.    in \cite{Neu}, Neufeld and Nowakowski showed that $c(G) - 1 \le c'(G) \le c(G)$, where $c(G)$ and $c'(G)$ are the minimum number of cops needed to guarantee capture the robber in the passive and active games, respectively.  

In \cite{OO}, Offner and Ojakian look at designating particular individuals (cops or robber) as active or flexible.  If an individual is active, then they must move to an adjacent vertex on their turn.  If an individual is flexible, then they can either move to an adjacent vertex or pass on their turn.   Under this model, we can, prior to the game, designate a particular subset of cops as active (leaving the remaining cops flexible) and designate whether the robber is active or flexible.  

We are interested in either all of the cops being flexible or all of the cops being active.  We let $c_{ff}(G)$, $c_{fa}(G)$ denote the cop-number of when all of the cops are flexible and the robber is either flexible or active, respectively.   Similarly, $c_{af}(G)$ and $c_{aa}(G)$ denote the corresponding cop-numbers where all of the cops are active.

Note that giving cops the ability to pass can only make it easier for the cops to capture the robber, while giving the robber the ability to pass can only make it more difficult for the cops.  Therefore,  $c_{fa}(G) \le c'(G) \le c_{aa}(G) \le c_{af}(G)$ and $c_{fa} \le c_{ff}(G)  = c(G) \le c_{af}(G)$.

In the game of Zombies and Survivor was introduced in \cite{FHMP}.  In this variation of Cops and Robber, the zombies take on the role of the cops, while the survivor takes on the role of the robber.    The zombies are active, with the added restriction that each zombie must move closer to the survivor.  That is, each zombie must move along a geodesic connecting himself and the survivor.  The survivor acts exactly like a flexible robber.  We let $\zombie (G)$ denote the minimum number of zombies required to guarantee the capture of the survivor.  It follows that  $c_{af}(G) \le \zombie (G)$, since the game of zombies and survivor is equivalent to active cops versus a flexible robber in which additional restrictions are placed on the cops.

\subsection{Retracts and Copnumber}

A {\it retraction} of $G$ to its subgraph $H$ is a homomorphism $f:G \rightarrow H$ such that $f$ is the identity function when restricted to $H$.    The subgraph $H$ in this case is referred to as a {\it retract} of $G$.  We can also define a {\it weak retraction} based on the idea of a {\it weak homomorphism}.  While a homomorphism maps adjacent vertices in $G$ to adjacent vertices in $H$, a weak homomorphism allows for adjacent vertices in $G$ to be mapped to the same vertex in $H$.    Based on a weak homomorphism, weak retraction and weak retract are defined accordingly.  Obviously, every retraction is also a weak retraction, and  if $G$ is a reflexive graph, every weak retraction is also a retraction.  

It is well known that $c(H) \le c(G)$ whenever $H$ is a retract of a reflexive graph $G$ \cite{Ber}.  However, given that we want to work with simple graphs, we restate this result as $c(H) \le c(G)$ whenever $H$ is a weak retract of a simple graph $G$.    In Section \ref{sec:trees}, we show that similar results hold for the parameters $c_{aa}$ and $c_{af}$ when $H$ is a retract of $G$.

\subsection{Cartesian Product}
The Cartesian product $G \by H$ of graphs $G$ and $H$ has vertex set $V(G \by H) = V(G) \times V(H)$, with the edge set $E(G \by H)$ consisting of edges $(u,v)(x,y)$ such that $u=x$ and $vy \in E(H)$, or $ux \in E(G)$ and $v=y$.    It is well known that (1) the Cartesian product of graphs $G$ and $H$ is connected if and only if both $G$ and $H$ are connected, and (2) the distance between any pair of vertices $(u,v)$ and $(x,y)$ in $G \by H$ is given by $d_{G \by H}( (u,v),(x,y)) = d_G(u,x) + d_H(v,y)$.  

For a collection of  graphs, $G_1, \ldots , G_k$ we accordingly define the Cartesian product  $\displaystyle{ \by_{i=1}^k G_i}$ to have vertex set  $V\left ( \by_{i=1}^k G_i \right) = \{(x_1,  \ldots, x_k): x_i \in V(G_i) , 1 \le i \le k\}$, where $d_{ \by_{i=1}^k G_i}((x_1, \ldots x_k), (y_1, \ldots y_k) )= \sum_{i=1}^k d_{G_i} (x_i,y_i)$.  We define $d_{G_i} (x_i,y_i)$ to be the distance between $(x_1, \ldots x_k)$ and  $(y_1, \ldots y_k)$ in the $i^{th}$ coordinate.

Suppose we are playing a game of zombies and survivors on the graph product graph $ \by_{i=1}^k G_i$, which we will call $G$.   When a player moves from vertex $(x_1, \dots, x_k)$ to an adjacent vertex $(y_1, \ldots, y_k)$ we note that for some $i \in \{1, \ldots , k\}$,  $x_i y_i \in E(G_i)$ and $x_j = y_j$ for all $j \neq i$.  In this case, we will say that the player has moved in coordinate $i$.   We also say that the projection of the player in $G_i$ has moved to an adjacent vertex (in $G_i$),  and for any $j \neq i$, the projection of the player in $G_j$ has passed (in $G_j$).    As a result, we can think of the game of zombies and survivors in $G$ corresponding to  {\it shadow games} in each of $G_1, \ldots , G_k$.   Since each zombie will pass in $k-1$ of the shadow games in each round, the shadow games do not  adhere to the rules for the Zombies and Survivors game.  They do, however, follow the rules of the traditional game of Cops and Robbers.

It was shown in \cite{MM} that the cop-number of the Cartesian product of $k$ non-trivial trees is $\left \lceil (n+1)/2 \right \rceil$. That is, $c(\by_{i=1}^k T_i ) = \left \lceil (n+1)/2 \right \rceil$.    For the active game,   it was was shown in \cite{Neu} that $c'(\by_{i=1}^n T_i) = \left \lceil n/2 \right \rceil$.     Finally, for the game in which all of the players are active, it was shown in \cite{OO} that for the n-dimensional hypercube, $Q_n$ (the Cartesian product of $n$ paths of length one), $c_{aa}(Q_n) = \left \lceil 2n/3 \right \rceil$.  

In Section \label{sec:trees}, we show that the zombie-number for the Cartesian product of $k$ non-trivial trees is $ \left \lceil 2n/3 \right \rceil$.  In doing so, we use techniques similar to those in to obtain results listed above.  This includes considering the parity of the distance between zombie and survivor as in  \cite{MM},  and using the ``home coordinate" approach taken in \cite{OO}.

\section{Zombie Number of the Cartesian Products of Trees}\label{sec:trees}

In this section, we show that for any set $\{T_1, T_2, \ldots , T_n\}$ of $n$ non-trivial trees,  $\zombie (\by_{i=1}^n T_i) = \lceil 2n/3 \rceil$.   We begin by establishing a more general result regarding retracts.

\begin{lemma}\label{retract}
If $G$ is a simple graph and $H$ is a retract of $G$, then $c_{aa}(H) \le c_{aa}(G)$ 
and $c_{af}(H) \le c_{af}(G)$. \end{lemma}

We omit the proof, as it is exactly the argument used in \cite{Ber} to prove that $c(H) \le c(G)$ for any weak retract $H$ of  $G$.  When $H$ is a weak retract of $G$, the associated weak homomorphism allows adjacent vertices in $G$ to be mapped to the same vertex in $H$.  Thus, given a move between two distinct vertices in $G$, the image of that move under the retraction could be a pass in $H$.  However, since $G$ is simple, a homomorphism would translate a move between two distinct vertices in $G$ into a move between distinct vertices in the retract.

We now show that $ \lceil 2n/3 \rceil$ is a lower bound on $\zombie (\by_{i=1}^n T_i)$.  This is done by demonstrating that the Cartesian product of $n$ trees has a retract isomorphic to $Q_n$.

\begin{lemma}\label{lowerbound}
If $\{T_1, \ldots , T_n\}$ is a set of $n$ non-trivial trees, then $\zombie \left ( \by_{i=1}^n \, T_i \right ) \ge \left \lceil 2n/3 \right \rceil$.
\end{lemma}

\begin{proof}
Consider any set of non-trivial trees $\{T_i:i=1, \ldots , n\}$.  Let $G = \by_{i=1}^n T_i$.  For each $i=1, \ldots , n$, let $x_iy_i$ be an edge in $T_i$.  Let $H$ be the subgraph of $G$ induced by $\{x_i,y_i : i=1, \ldots , n \}$.  It follows that $H \cong Q_n$.  

Next, suppose a proper 2-colouring is applied to each $T_i$.  For each $i = 1, \ldots  n$, let $f_i:V(T_i) \rightarrow \{x_i,y_i\}$ be defined as follows: $f_i$ maps each vertex in $T_i$ to the vertex in $\{x_i, y_i\}$  with which it shares a colour class.  Note that $f_i$ is the identity when restricted to $\{x_i, y_i\}$.
It follows that $F:V(G) \rightarrow V(H)$, defined by $F((z_1, \ldots , z_n) ) = (f_1(z_1), \ldots , f_n(z_n))$ is a retraction from $G$ to $H$.  

Since $H$ is a retract of $G$ and $H \cong Q_n$, it follows from Lemma \ref{retract} that $c_{af}(Q_n) \le c_{af}(G)$.  Furthermore, it was shown in \cite{OO} that $c_{aa}(Q_n) = \lceil 2n/3 \rceil$.  Since $c_{aa}(Q_n) \le c_{af}(Q_n) \le  c_{af}(G) \le \zombie (G)$, it follows that $\zombie (G)  \ge \lceil 2n/3 \rceil $.

\end{proof}


We now have the task of showing that  $\zombie \left ( \by_{i=1}^n \, T_i \right ) \le \left \lceil 2n/3 \right \rceil$.  This is proved by induction on $n$ in Theorem \ref{upperbound}.  A series of results precede this, beginning with the base cases of $n=2$ and $n=3$.  

\begin{lemma}\label{3trees}
For any non-trivial trees $T_1$, $T_2$ and $T_3$, $\zombie (T_1 \by T_2 \by T_3) = 2$.
\end{lemma}

\begin{proof}
Let $G = {T_1 \by T_2 \by T_3}$.  We begin the game by placing zombies $Z_1$ and $Z_2$ on adjacent vertices in $G$.  It follows that the vertices occupied by $Z_1$ and $Z_2$ differ in exactly one coordinate, and the distance between the vertices in that particular coordinate is one.  The survivor $S$ then chooses a vertex in $G$.  For the remainder of the game the each zombie  moves according to the following rule:  the zombie moves in the coordinate in which the distance between himself and the survivor is greatest.  To break ties,  $Z_1$ choose the leftmost coordinate, while $Z_2$ chooses the rightmost coordinate.

We note that  each of the shadow games is taking place in a tree.  Furthermore, each zombie is moving to an adjacent vertex in exactly one of the shadow games in each round.  As a result, after some finite number of moves, the each zombie will capture the survivor in at least one shadow game.  Suppose we are at the mid-point of a round in which this has occurred, but the zombies have not yet captured the survivor in $G$.

We now introduce a  distance vector  $v_1 = [a_1,a_2,a_3]$, where $a_j$ denotes the distance between $Z_1$ and $S$ in coordinate $j$.  Distance vector $v_2 = [b_1,b_2,b_3]$ representing the distances between $Z_2$ and $S$ is defined similarly.   For each of $v_1$ and $v_2$, at least one of the entries in the vector is 0.  It follows that the other entries are all at most 1, given the rule that the zombies must move in the coordinate where the distance is greatest.  That, together with the rule for breaking ties, tells us that $v_1$ is either $[0,1,1]$ or $[0,0,1]$, and $v_2$ is either $[1,1,0]$ or $[1,0,0]$.    Due to the rules for the zombies' play, the possible values of $v_1$ and $v_2$ will remain the same when measured at the mid-point of each subsequent round, with the only exception occurring when $S$ is captured and either $v_1$ or $v_2$ is the zero-vector.

Let $|v_1|$ and $|v_2|$ denote the sum of their respective entries.   We claim that, due to the choice of initial positions for $Z_1$ and $Z_2$, $|v_1| + |v_2|$ is odd.   

\medskip

\noindent {\it Proof of Claim:}  For convenience, we will associate each player with the vertex they occupy and let $d(x,y) = d_G(x,y)$.  It follows that at the end of the first round, $|d(Z_1,S) - d(Z_2, S)| = 1$.  Therefore, $d(Z_1,S) + d(Z_2, S)$ is odd at the end of round 1.  Now, since each factor of $G$ is a tree, when the survivor moves to an adjacent vertex, his distance from $Z_1$ changes by +1 or -1 in exactly one coordinate.  Therefore, $d(Z_1,S)$ either increases or decreases by one.  Similarly for $d(Z_2,S)$.  Hence, $d(Z_1,S) + d(Z_2, S)$ will remain odd if $S$ either moves to an adjacent vertex or passes.  Now, when the zombies move, each of $d(Z_1,S)$ and $d(Z_2,S)$ will decrease by exactly one.  Hence, at any point in the game, the sum $d(Z_1,S) + d(Z_2, S)$ is odd.  Hence, $|v_1| + |v_2|$ is odd. 

\medskip

With the claim proved, we may assume, without loss of generality, that we are at the mid-point of some round where $v_1 = [0,0, 1]$ and $v_2 = [1,1,0]$.   
We now consider the survivor's play.  If the survivor moves sufficient number of times in the 3rd coordinate, $Z_1$ will capture $S$.  This is due to the fact that there will be sufficient play in the shadow game on $T_3$ for $Z_1$ to capture $S$ in $T_3$.  It follows that, at some point in the game, $S$ must move indefinitely in the first two coordinate (and cannot pass).

This means  $Z_2$ will capture $S$ in one of the shadow games in $T_1$ or $T_2$. Due to the tie-breaking rule, $Z_2$ captures $S$ in the shadow game on $T_2$.   It follows that $v_2 = [1,0,0]$.  Since $|v_1|+|v_2|$ is odd, it follows that $v_1 = [0,1,1]$ or $v_1 = [0,0,0]$.  The former, however, is impossible since the distance between $Z_1$ and $S$ does not increase from the mid-point of one round to the next.  Hence, $v_1 = [0,0,0]$ and the survivor has been captured.  Hence $z(T_1 \by T_2 \by T_3) \le 2$.

Since $z(G) \ge c(G)$ and $c(T_1 \by T_2 \by T_3)  = \left \lceil (3+1)/2 \right \rceil =2$, the result follows.
\end{proof}

\begin{corollary}\label{2trees} For any non-trivial trees $T_1$ and $T_2$, $\zombie (T_1 \by T_2) = 2$.
\end{corollary}

\begin{proof}
Since $c(T_1 \by T_2)  = \left \lceil (2+1)/2 \right \rceil =2$, it follows that $\zombie (T_1 \by T_2) \ge 2$.  Furthermore, $\zombie (T_1 \by T_2) \le 2$ since $\zombie (T_1 \by T_2 \by T_3) \le 2$ will also hold when $T_3$ is the tree on a single vertex.
\end{proof}

We now introduce the {\it home strategy} for the zombies.  We begin by a particular vertex coordinate to each zombie.  This is referred to as the zombie's home coordinate.  For a given zombie, the home strategy would be as follows:  at the beginning of each round, the zombie will compare his position with the survivor's position.  Starting with the home coordinate, if their positions differ in that coordinate, then the zombie moves in that coordinate.  If not, the zombie  compares the next coordinate to the right and repeats the process.   Note that the coordinate  to the immediate right of the $k^{th}$ coordinate is defined to be the first coordinate.   As a result, the zombie moves in the  coordinate where the positions first differ, starting at the home coordinate and moving to the right.  Obviously, if they differ in no coordinate, the zombie has captured the survivor and the game is over.

We say that coordinate $i$ is {\it within reach} of a zombie if that zombie's position and survivor's positions are identical in all coordinates from the home coordinate to coordinate $i$, inclusive, as you move to the right of the home coordinate.
We  define the {\it reach} of the zombie to be the number of  coordinates that are within reach of that zombie.  Unless otherwise indicated, we measure the reach of each zombie at the midpoint of each round (following the zombies' move).   

We note that for each zombie, its reach is non-decreasing when it moves according to the home strategy.

\begin{lemma}\label{moretrees}
Let $G$ be the Cartesian product of any collection of $n$ non-trivial trees, such that $n \equiv 0$ or $2 \pmod 3$.  It follows that a collection $\lceil 2n/3 \rceil$ zombies are sufficient to win on $G$, provided exactly half of the zombies are an even distance from the survivor at the end of the first round.
\end{lemma}

\begin{proof}
By Lemma \ref{3trees} and Corollary \ref{2trees} we see that result holds when $k=2$ or $n=3$.  We now proceed by induction.

Consider a set of non-trivial trees, $\{T_i : i=1, \ldots , n\}$, where $n=3\ell - 1$ or $n = 3\ell$ for some $\ell \ge 2$.  Let $G = \by_{i=1}^n T_i$.  Since $\lceil 2n/3 \rceil = 2\ell$, we consider a collection of  $2\ell$ zombies:  $Z_1, \ Z_2, \ldots , Z_{2\ell}$.  Each of the zombies selects an initial position in the game.  The survivor, $S$, then selects an initial position.  Assume that exactly half of the zombies are an even distance from $S$.  Without loss of generality, assume $Z_1, \ldots , Z_\ell$ are those at an even distance from $S$.   As in \ref{3tress}, this guarantees that for each $i=1, \ldots \ell$,  $d(Z_i,S) + d(Z_{\ell +i},S)$ is odd for the remainder of the game.

For each $i= 1, \ldots , \ell$, zombies $Z_i$ and $Z_{i+\ell}$ are both assigned coordinate $3i-2$ as their home coordinate.   Each zombie will initially play using the home strategy. 
Using the home strategy, we know that the reach of each zombie is non-decreasing. If we consider the total reach of the zombies to be the sum of the individual reaches, then the survivor wins only if the total reach remains constant for an indefinite number of rounds.  Furthermore, if the survivor either passes or moves in a coordinate that is not within reach of a zombie, $Z$,  in a sufficient number of rounds, the reach of $Z$ will increase.   Therefore, for the survivor to win, there must be some indefinite number of rounds in which the survivor does not pass and only moves in  coordinates that are within reach of every zombie.

Therefore, we may assume that at some point in the game, there is a coordinate that is within reach of every zombie.  Say it is the $w^{th}$ coordinate. If $n-2 \le w \le n$, then at most two coordinates, the $n-1^{st}$ and the $n^{th}$, are not within reach of $Z_1$.  If $w < n-2$, then $3t-2\le w \le 3t$ for some $t$ such that $1 \le t <\ell$, and at most two coordinates, $3t-1$ and $3t$,  are not within reach of $Z_{t+1}$.    It follows that, at some point in the game, for some $i$ such that $1 \le i \le \ell$, zombie $Z_i$ has at least $n-2$ coordinates within reach.   It follows that $Z_{\ell +i}$ also has at least $n-2$ coordinates within reach, and those coordinates possibly not within reach of $Z_i$ are the same as those possibly not within reach of $Z_{\ell +i}$.  

Without loss of generality, suppose both $Z_1$ and $Z_{\ell +1}$ have the first through the $(n-2)^{nd}$ coordinates within reach.    We may assume that the zombies have just completed their move, and it is the midpoint of some round.  The zombies $Z_1$ and $Z_{\ell +1}$ now play according to the following rule:  if $S$ moves in one of the first $n-3$ coordinates, $Z_1$ and $Z_{\ell +1}$ will both mimic that move in the next round, otherwise, $Z_1$ and $Z_{\ell +1}$ will move in the last three coordinates according to their winning strategy on the product $T_{n-2} \by T_{n-1} \by T_{n}$.   We note that at the midpoint of any round, exactly one of $Z_1$ and $Z_{\ell +1}$ is even distance from $S$.  Since the vertices of $Z_1$, $Z_{\ell+1}$ and $S$ are the same in the first $n-3$ coordinates, then in the shadow game on $T_{n-2} \by T_{n-1} \by T_{n}$, exactly one of  $Z_1$ and $Z_{\ell +1}$ is even distance from $S$.  This will remain the case at any point in the game, no matter how $S$ moves (or passes).  It follows that, if $S$ makes a sufficient number of moves in the last three coordinates, or passes, in subsequent rounds, $Z_1$ or $Z_{\ell +1}$ will eventually apprehend the robber.    We may, therefore, assume that $S$ makes an indefinite number of moves in the first $n-3$ coordinates.  

Meanwhile, each zombie in  ${\cal Z} = \{Z_i, Z_{\ell + i}: 2 \le i \le  \ell\}$
adopts $n-2$ as their new home coordinate and uses the home strategy.   It follows that, since $S$ makes an indefinite number of  moves in the first $n-3$ coordinates, every zombie in $\cal Z$ will eventually have the last three coordinates within reach.   Once this is accomplished, the zombies of $\cal Z$ adopt the following strategy: if $S$ moves in  one of the last three coordinates, each zombie  in $\cal Z$ mimics that move.  Otherwise, the $2\ell-2$ zombies move in the first $n-3$ coordinates, using their winning strategy on $\by_{i=1}^{n-3} T_i$.  Again, this is possible since exactly half of the zombies are an even distance from $S$ in the shadow game on $\by_{i=1}^{n-3} T_i$.  Since  $S$ makes an indefinite number of moves in the first $n-3$ coordinates, a zombie in $\cal Z$ eventually captures $S$.
\end{proof}

\begin{theorem}\label{upperbound}
For any set of $n$ non-trivial trees $\{T_i : i=1, \ldots , n\}$, where $n\ge 2$,  $\zombie ( \by_{i=1}^n T_i) = \left \lceil \frac{2n}{3} \right \rceil$.
\end{theorem}

\begin{proof}
Let $G =  \by_{i=1}^n T_i$, where $n \ge 2$ and each $T_i$ is a non-trivial tree.  We now consider two cases, based on the value of $n \pmod 3$.  

\medskip

\noindent {\it Case 1:  $n \equiv 0$ or $2 \pmod 3$. }   It follows that for some $\ell \ge 1$, $n = 3 \ell -1$ or $n=3\ell$, and $\left \lceil \frac{2n}{3} \right \rceil = 2\ell$.    Choose vertices $x$ and $y$ such that $x$ and $y$ are adjacent in $G$.   Given $2\ell$ zombies, begin the game by placing half the zombies on $x$ and half the zombies on $y$.  As seen in Lemma \ref{2trees}, any choice of initial position by $S$ will result in exactly half the zombies being an even distance from $S$.  It follows from Lemma \ref{moretrees} that $z(G) \le \left \lceil \frac{2n}{3} \right \rceil$.

\medskip

\noindent {\it Case 2:  $n \equiv 1  \pmod 3$. }    It follows that $\left \lceil \frac{2n}{3} \right \rceil = 2\ell+1$ for some $\ell \ge 1$.    Choose vertices $x$ and $y$ such that $x$ and $y$ are adjacent in $G$.    The game begins by placing zombies $Z_1, \ldots , Z_\ell$ and $Z_{2\ell +1}$ on vertex $x$, and zombies $Z_{\ell +1}, \ldots , Z_{2\ell}$ on $y$.  For each $i = 1, \ldots , \ell$, assign $Z_i$  coordinate $3i-2$ and  assign $Z_{\ell +i}$  coordinate $3i-1$ as their respective home coordinates.  Finally, $Z_{2\ell +1}$ is assigned coordinate $n$ as its home coordinate.  As in the proof of Lemma \ref{moretrees}, there is a coordinate $w$ that, at some point in the game, is within reach of all the zombies.  

If $w = n$ or $w=n-1$, then either $Z_1$ or $Z_{2\ell +1}$, respectively, has captured $S$.   If $w = 3t+1$ or $w=3t$ for some $t$ such that $1 \le t < \ell$, then either $Z_{\ell + t}$ or $Z_{t+1}$, respectively, has captured $S$.  We may, therefore, assume that $w = 3t-1$, where $1 \le t \le \ell$.  Also assume that $S$ has not been captured.
If $t = \ell$, then $Z_{2\ell +1}$ has every coordinate except  $(3\ell)^{th}$ coordinate within its reach.  If $1 \le t <\ell$, then $Z_{t+1}$ has every coordinate with the exception of $3t$ within its reach.  

Without loss of generality, assume we renumber zombies, tress and  coordinates so that Let $Z_{2\ell -1}$ is the zombie from $\{Z_1, \ldots , Z_\ell\} \cup \{Z_{2\ell +1}\}$ that has all but the $n^{th}$ coordinate within its reach.   We now assume that $Z_{2\ell -1}$ moves according to the home strategy, with its home coordinate being the first coordinate.   It follows that $S$ must move indefinitely in the first $n-1$ coordinate.  

Meanwhile, zombies $Z_1, \ldots Z_{2\ell}$ each move in the $n^{th}$ coordinate until they have all captured the survivor in the corresponding shadow game in $T_n$.  Once this is accomplished, they each mimic the survivor's move whenever he moves in the $n^{th}$ coordinate.  When $S$ moves in the first $n-1$ coordinates, or passes, the zombies $Z_1, \ldots Z_{2\ell}$ 
move in the first $n-1$ coordinates according to their winning strategy in Lemma \ref{moretrees}.  Since $Z_1, \ldots Z_{\ell}$ started on vertex $x$, and  $Z_{\ell +1}, \ldots Z_{2\ell}$ started on $y$, exactly half will be an even distance from $S$ in the shadow game on $\by_{i=1}^{n-1} T_i$.   Therefore, by Lemma \ref{moretrees}, the survivor will be captured.

\medskip

Hence, $\zombie ( \by_{i=1}^n T_i) \le \left \lceil \frac{2n}{3} \right \rceil$, and by Lemma \ref{lowerbound}  $\zombie ( \by_{i=1}^n T_i) = \left \lceil \frac{2n}{3} \right \rceil$
\end{proof}

\begin{corollary}
For any set of $n$ non-trivial trees $\{T_i : i=1, \ldots , n\}$, where $n\ge 2$,  $c_{aa} ( \by_{i=1}^n T_i)  = c_{af} ( \by_{i=1}^n T_i) =\left \lceil \frac{2n}{3} \right \rceil$.
\end{corollary}

\begin{proof}
Since $ \left \lceil \frac{2n}{3} \right \rceil = c_{aa} ( Q_n ) \le c_{aa} ( \by_{i=1}^n T_i) \le c_{af} ( \by_{i=1}^n T_i) \le \zombie ( \by_{i=1}^n T_i) = \left \lceil \frac{2n}{3} \right \rceil$, equality is satisfied through the expression. 
\end{proof}





\end{document}